\documentclass[11pt]{amsart}
\usepackage{amsfonts}
\usepackage{amsmath, amsthm}

\usepackage{comment}
\usepackage{tensor}

\numberwithin{equation}{section}

\newtheorem*{Theorem}{Theorem}
\newtheorem{Lemma}{Lemma}[section]

\newcommand{\cC}{\mathcal{C}}
\renewcommand{\L}{\mathcal{L}}

\newcommand{\R}{\mathbb{R}}
\newcommand{\cR}{\mathcal{R}}
\renewcommand{\S}{\mathbb{S}}
\newcommand{\cS}{\mathcal{S}}

\newcommand{\del}{\partial}

\renewcommand{\phi}{\varphi}
\newcommand{\grad}{\nabla}
\renewcommand{\epsilon}{\varepsilon}
\newcommand{\II}{\operatorname{II}}

\newcommand{\Lip}{\mathrm{Lip}}

\title[Eigenvalue maximization for surfaces of revolution]{Eigenvalue maximization for surfaces of revolution with prescribed boundary}
\author{Sinan Ariturk}
\date{}

\begin{document}

\begin{abstract}
Fix two parallel circles in $\R^3$ centered about a common axis.
Among surfaces of revolution immersed in $\R^3$ whose boundary is given by these circles, there is one which maximizes the first Dirichlet eigenvalue.
If the circles are sufficiently close together, then this surface is unique.
\end{abstract}

\maketitle

\section{Introduction}
Fix two parallel circles $P$ and $Q$ in $\R^3$ centered about a common axis.
Let $\cS$ be the set of all compact smoothly immersed surfaces of revolution in $\R^3$ with two boundary components, given by $P$ and $Q$.
For a surface $\Sigma$ in $\cS$, let $\Delta_\Sigma$ be the Laplace-Beltrami operator on $\Sigma$.
Denote the Dirichlet eigenvalues of $-\Delta_\Sigma$ by
\[
	0 < \lambda_1(\Sigma) \le \lambda_2(\Sigma) \le \lambda_3(\Sigma) \le \ldots
\]
We view the first eigenvalue as a functional $\lambda_1: \cS \to \R$.

\begin{Theorem}
There is an embedded surface $\Sigma^*$ in $\cS$ such that
\[
	\lambda_1(\Sigma^*) = \sup \Big\{ \lambda_1(\Sigma) : \Sigma \in \cS \Big\}
\]
There is an $\epsilon >0$ which depends on $P$ but not on $Q$, such that, if the distance between the circles $P$ and $Q$ is less than $\epsilon$, then there is exactly one maximizing surface $\Sigma^*$ which is connected.
\end{Theorem}

We remark that there are compact smoothly embedded surfaces in $\R^3$ with two boundary components, given by $P$ and $Q$, which are not surfaces of revolution and have first Dirichlet eigenvalue larger than $\lambda_1(\Sigma^*)$.
This can be proven with Berger's variational formulas \cite{B}.

This problem resonates with the Rayleigh-Faber-Krahn inequality, which states that the flat disc has smaller first Dirichlet eigenvalue than any other domain in $\R^2$ with the same area \cite{F} \cite{K}.
Hersch proved that the canonical metric on $\S^2$ maximizes the first non-zero eigenvalue among metrics with the same area \cite{H}.
Li and Yau showed the canonical metric on $\mathbb{RP}^2$ maximizes the first non-zero eigenvalue among metrics with the same area \cite{LY}.
Nadirashvili proved the same is true for the flat equilateral torus, whose fundamental parallelogram is comprised of two equilateral triangles \cite{N}.
It is not known if there is such a maximal metric on the Klein bottle, but Jakobson, Nadirashvili, and Polterovich showed there is a critical metric \cite{JNP}.
El Soufi, Giacomini, and Jazar proved this is the only critical metric on the Klein bottle \cite{EGJ}.

On a compact orientable surface, Yang and Yau obtained upper bounds, depending on the genus, for the first non-zero eigenvalue among metrics of the same area \cite{YY}.
Li and Yau extended these bounds to compact non-orientable surfaces \cite{LY}.
However, Urakawa showed that there are metrics on $\S^3$ with volume one and arbitrarily large first non-zero eigenvalue \cite{U}.
Colbois and Dodziuk extended this to any manifold of dimension three or higher \cite{CD}.

For a closed compact hypersurface in $\R^{n+1}$, Chavel and Reilly obtained upper bounds for the first non-zero eigenvalue in terms of the surface area and the volume of the enclosed domain \cite{Ch,R}.
Abreu and Freitas proved that for a metric on $\S^2$ which can be isometrically embedded in $\R^3$ as a surface of revolution, the first $\S^1$-invariant eigenvalue is less than the first Dirichlet eigenvalue on a flat disc with half the area \cite{AF}.
Colbois, Dryden, and El Soufi extended this to $O(n)$-invariant metrics on $\S^n$ which can be isometrically embedded in $\R^{n+1}$ as hypersurfaces of revolution \cite{CDE}.
%Colbois, Dryden, and El Soufi obtained upper bounds in terms of the intersection index, which counts the number of points of intersection with transverse planes \cite{CDE2}.

We conclude this section by reformulating the theorem.
Let $R$ be the maximum of the radii of $P$ and $Q$.
Let $D_R$ be a disc in $\R^2$ with radius $R$.
If $P$ and $Q$ are not coplanar, then there is a surface in $\cS$ comprised of two flat discs and its first Dirichlet eigenvalue is $\lambda_1(D_R)$.
If $P$ and $Q$ are coplanar, then there is an annulus in $\cS$ and its first Dirichlet eigenvalue is larger than $\lambda_1(D_R)$.
In either case,
\[
	\sup \Big\{ \lambda_1(\Sigma) : \Sigma \in \cS \Big\} \ge \lambda_1(D_R)
\]
Note if $\Sigma$ is a surface in $\cS$, then either there is a connected component of $\Sigma$ which is diffeomorphic to a cylinder or there are two connected components of $\Sigma$ which are diffeomorphic to discs.

\begin{Lemma}
\label{topodisc}
Let $\Sigma$ be a surface in $\cS$.
Assume that $\Sigma$ has a connected component which is diffeomorphic to a disc.
Then
\[
	\lambda_1(\Sigma) \le \lambda_1(D_R)
\]
\end{Lemma}

In light of this lemma, we only need to consider surfaces in $\cS$ which are connected.
Naturally, we can restate the theorem in terms of curves in a half-plane.
Fix a plane in $\R^3$ containing the axis of symmetry of $P$ and $Q$.
Identify $\R^2$ with this plane in such a way that the axis of symmetry is identified with
\[
	\Big\{ (x,y) \in \R^2 : x =0 \Big\}
\]
Define
\[
	\R^2_+ = \Big\{ (x,y) \in \R^2 : x > 0 \Big\}
\]
Let $p$ be the point where the circle $P$ intersects $\R^2_+$.
Similarly, let $q$ be the point where $Q$ intersects $\R^2_+$.
Let $\cC$ be the set of smooth, regular curves $\gamma: [0,1] \to \R^2_+$ with $\gamma(0)=p$ and $\gamma(1)=q$.
If $\gamma$ is a curve in $\cC$, write $\gamma=(F,G)$ and define
\[
	\lambda_1(\gamma) = \min \Bigg\{ \frac{ \int_0^1 \frac{|w'|^2 F}{|\gamma '|} \,dt}{ \int_0^1 |w|^2 F | \gamma' | \,dt} : w \in C_0^\infty(0,1) \Bigg\}
\]
If $\Sigma$ is a connected surface in $\cS$ and $\gamma$ is a curve in $\cC$ which parametrizes $\Sigma \cap \R^2_+$, then
\[
	\lambda_1(\Sigma) = \lambda_1(\gamma)
\]

\begin{Lemma}
\label{exreg}
Assume that
\[
	\sup \Big\{ \lambda_1(\gamma) : \gamma \in \cC \Big\} > \lambda_1(D_R)
\]
Then there is a simple curve $\alpha$ in $\cC$ such that
\[
	\lambda_1(\alpha) = \sup \Big\{ \lambda_1(\gamma) : \gamma \in \cC \Big\}
\]
\end{Lemma}

This lemma implies the existence part of the theorem.
To prove this lemma, we take a maximizing sequence of curves.
We first prove that the curves in the sequence have bounded length and stay in a compact subset of $\R^2_+$.
Reparametrizing the curves by arc length, we obtain a subsequence which converges uniformly to a Lipschitz continuous curve.
We can extend the eigenvalue functional to Lipschitz curves.
By an upper semi-continuity argument, the limit curve is maximizing.
It then remains to prove that the limit curve is smooth.
The first step is to show that the maximizing curve is continuously differentiable.
The proof of this will be based on the following geometric fact: if the maximizing curve intersects a small circle at two points, then it must stay inside the circle.
Otherwise, reflecting the curve inside the circle would result in a curve with a greater eigenvalue.
Once this amount of regularity is established, we consider variations of the curve and prove that the eigenvalue is differentiable along these variations.
This enables us to derive a differential equation which the maximizing curve solves weakly.
Then it is straightforward to use the differential equation to see that the curve is smooth.

\begin{Lemma}
\label{unique}
There is an $\epsilon >0$ which depends on $p$ but not on $q$, such that, if the distance between $p$ and $q$ is less than $\epsilon$, then the curve $\alpha$ in $\cC$ which satisfies
\[
	\lambda_1(\alpha) = \sup \{ \lambda_1(\gamma) : \gamma \in \cC \}
\]
is unique, up to reparametrization.
\end{Lemma}

This lemma establishes the uniqueness part of the theorem.
To establish this lemma, we look at the differential equation which describes curves which are critical points of the eigenvalue functional.
We view this as an initial value problem originating at $p$ and apply the shooting method.
This yields a family of critical curves with one endpoint at $p$, parametrized by their tangent vector $v$ at $p$ and their eigenvalue $\lambda$.
Let $q(v,\lambda)$ denote the other endpoint.
We prove that the restriction of the function $q$ to large $\lambda$ is a diffeomorphism onto a small punctured neighborhood of $p$.
This implies the lemma.

\section{Existence}

In this section we will prove Lemma \ref{topodisc} and a low regularity version of Lemma \ref{exreg}.
We first extend the domain of the functional to Lipschitz curves.
For real numbers $c<d$, let $\Lip_0(c,d)$ denote the set of Lipschitz continuous functions $w:[c,d] \to \R$ which vanish at the endpoints.
For a Lipchitz continuous function $\gamma: [c,d] \to \R^2_+$ which is not constant, define
\[
	\lambda_1(\gamma) = \inf \bigg\{ \frac{ \int_c^d \frac{| w' |^2 F}{|\gamma'|} \,dt}{\int_c^d | w |^2 F |\gamma'| \,dt} : w \in \Lip_0(c,d) \bigg\}
\]
We remark that the quotient may be infinite, for some choices of $\gamma$ and $w$.
However, it follows from Lemma \ref{arcparam} that $\lambda_1(\gamma)$ is finite for any non-constant Lipschitz continuous $\gamma:[c,d] \to \R^2_+$.

Let $\cR$ be the set of Lipschitz continuous functions $\gamma: [0,1] \to \R^2_+$ with $\gamma(0)=p$ and $\gamma(1)=q$.
Note that if $\gamma$ is a smooth curve in $\cC$, then this definition of $\lambda_1(\gamma)$ agrees with the previous one.
Define
\[
	\Lambda = \sup \Big\{ \lambda_1( \gamma ) : \gamma \in \cR \Big\}
\]
In this section we prove that there is a curve $\alpha$ in $\cR$ such that
\[
	\lambda_1(\alpha) = \Lambda
\]
We take a maximizing sequence in $\cR$ and obtain a convergent subsequence.
Then we will see that the limit curve is maximizing.

In the following lemma, we will prove that reparametrizing a curve to have constant speed does not decrease the eigenvalue.

\begin{Lemma}
\label{arcparam}
Let $\gamma: [c,d] \to \R^2_+$ be Lipschitz continuous.
Let $L$ be the length of $\gamma$.
Define $\ell:[c,d] \to [0,1]$ by
\[
	\ell(t) = \frac{1}{L} \int_c^t | \gamma'(u)| \,du
\]
There is a Lipschitz continuous function $\beta: [0,1] \to \R^2_+$ such that, for all $t$ in $[c,d]$,
\[
	\beta(\ell(t)) = \gamma(t)
\]
For almost all $t$ in $[0,1]$,
\[
	| \beta'(t) | = L
\]
Moreover,
\[
	\lambda_1(\beta) \ge \lambda_1(\gamma)
\]
\end{Lemma}

\begin{proof}
Define $\eta:[0,1] \to \R$ by
\[
	\eta(s) = \min \Big\{ t \in [c,d] : \ell(t)=s \Big\}
\]
Note that $\eta$ may not be continuous, but $\beta = \gamma \circ \eta$ is Lipschitz continuous, and for all $t$ in $[c,d]$,
\[
	\beta(\ell(t)) = \gamma(t)
\]
For almost all $t$ in $[0,1]$,
\[
	| \beta'(t) | = L
\]
Note that $\lambda_1(\beta)$ is finite.
Write $\gamma=(F_\gamma,G_\gamma)$ and $\beta=(F_\beta, G_\beta)$.
Now let $\epsilon>0$.
There is a function $w$ in $\Lip_0(0,1)$ which satisfies
\[
	\frac{ \int_0^1 \frac{| w' |^2 F_\beta}{|\beta'|} \,dt}{\int_0^1 | w |^2 F_\beta |\beta'| \,dt} < \lambda_1(\beta) + \epsilon
\]
Define $v = w \circ \ell$.
Then $v$ is in $\Lip(c,d)$.
Changing variables yields
\[
	\lambda_1(\gamma) \le \frac{ \int_c^d \frac{| v' |^2 F_\gamma}{|\gamma'|} \,dt}{\int_c^d | v |^2 F_\gamma |\gamma'| \,dt} = \frac{\int_0^1 \frac{| w' |^2 F_\beta}{|\beta'|} \,dt}{\int_0^1 | w |^2 F_\beta |\beta'| \,dt} < \lambda_1(\beta) + \epsilon \\
\]
\end{proof}

For positive numbers $a<b$, let $A_{a,b}$ be the concentric annulus in $\R^2$ centered about the origin with outer radius $b$ and inner radius $a$.
Let $\lambda_1(A_{a,b})$ be the first Dirichlet eigenvalue of the Laplace operator on $A_{a,b}$.

\begin{Lemma}
\label{Fbound}
Let $I$ be a compact interval and let $\gamma:I \to \R^2_+$ be Lipschitz.
Write $\gamma=(F,G)$.
If the image of $F$ is $[a,b]$, with $a<b$, then
\[
	\lambda_1(\gamma) \le \lambda_1(A_{a,b})
\]
\end{Lemma}

\begin{proof}
Let $\alpha: I \to \R^2_+$ be the Lipschitz function defined by $\alpha=(F,0)$.
For any function $w$ in $\Lip_0(I)$ which is not identically zero,
\[
	\frac{\int_I \frac{|w'|^2 F}{| \gamma'|} \,dt}{\int_I w^2 F | \gamma' | \,dt} \le \frac{\int_I \frac{|w'|^2 F}{| \alpha'|} \,dt}{\int_I w^2 F | \alpha' | \,dt}
\]
Therefore $\lambda_1(\gamma) \le \lambda_1(\alpha)$.
Let $[c,d]$ be a subinterval of $I$ such that the image of $\{c,d\}$ under $F$ is $\{a,b\}$.
Assume that $F(t)$ is in $(a,b)$ for every $t$ in $(c,d)$.
For simplicity, we assume that $F(c)=a$ and $F(d)=b$, since a symmetric argument can be used in the other case.
Define $F_\zeta:[c,d] \to \R$ by
\[
	 F_\zeta(t) = \min \{ F(s) : s \in [t,d] \}
\]
Define
\[
	U = \bigg\{ t \in [c,d] : F_\zeta(t) \neq F(t) \bigg\}
\]
By the Riesz sunrise lemma, $F_\zeta$ is constant over each connected component of $U$.
Let
\[
	J = [c,d] \setminus U
\]
Define a Lipschitz function $\zeta:[c,d] \to \R^2_+$ by $\zeta=(F_\zeta, 0)$.
Note that $\lambda_1(\zeta)$ is finite, by Lemma \ref{arcparam}.
Let $\epsilon>0$.
There is a function $w$ in $\Lip_0(c,d)$ such that
\[
	\frac{ \int_c^d \frac{| w' |^2 F_\zeta}{|\zeta'|} \,dt}{\int_c^d | w |^2 F_\zeta |\zeta'| \,dt} < \lambda_1(\zeta) +\epsilon
\]
Note that $| \zeta ' |$ is zero over $U$. 
Therefore $w$ is constant over $U$.
The isolated points of $J$ are countable, so at almost every point in $J$, the curve $\zeta$ is differentiable with $\zeta'=\alpha'$.
Then
\[
	\frac{ \int_c^d \frac{| w' |^2 F}{|\alpha'|}  \,dt}{\int_c^d | w |^2 F |\alpha'| \,dt} \le \frac{ \int_c^d \frac{| w' |^2 F_\zeta}{|\zeta'|}  \,dt}{\int_c^d | w |^2 F_\zeta |\zeta'| \,dt}
\]
Therefore $\lambda_1(\alpha) \le \lambda_1(\zeta)$.
Define $\beta: [0,1] \to \R^2_+$ by
\[
	\beta(t) = \Big( a(1-t) + bt , 0 \Big)
\]
Then $\lambda_1(\zeta) \le \lambda_1(\beta)$ by Lemma \ref{arcparam}.
Note that $\lambda_1(\beta) = \lambda_1(A_{a,b})$.
\end{proof}

In the following lemma, we bound the length of a curve $\gamma$ in $\cR$ in terms of $\lambda_1(\gamma)$.

\begin{Lemma}
\label{Lbound}
Let $\gamma$ be a curve in $\cR$.
Let $L$ be the length of $\gamma$.
Write $\gamma=(F,G)$.
If the image of $F$ is $[a,b]$, then
\[
	L \le \pi \sqrt{ \frac{b}{a \lambda_1(\gamma)} }
\]
\end{Lemma}

\begin{proof}
Let $\beta$ be the constant-speed reparametrization given by Lemma \ref{arcparam}.
Then $\lambda_1(\beta) \ge \lambda_1(\gamma)$.
Define $w: [0,1] \to \R$ by
\[
	w(t) = \sin(\pi t)
\]
Write $\beta=(F_\beta,G_\beta)$.
Then
\[
	\int_0^1 \frac{| w'(t) |^2 F_\beta(t)}{| \beta'(t)|} \,dt \le \frac{\pi^2 b}{L} \int_0^1 \cos^2(\pi t) \,dt
\]
Also,
\[
	\int_0^1 | w(t) |^2 F_\beta(t) | \beta'(t)| \,dt \ge La \int_0^1 \sin^2(\pi t) \,dt
\]
Now
\[
	\lambda_1(\gamma) \le \lambda_1(\beta) \le \frac{\pi^2b}{L^2a}
\]
\end{proof}

In the following lemma, we prove that there is a curve $\alpha$ in $\cR$ such that $\lambda_1(\alpha)=\Lambda$.
We take a maximizing sequence of curves in $\cR$.
By Lemma \ref{arcparam}, we may assume the curves in the sequence are parametrized by arc length.
Then Lemma \ref{Fbound}, Lemma \ref{Lbound} and the Arzela-Ascoli theorem imply that a subsequence converges uniformly to a curve in $\cR$.
By upper semi-continuity, the limit curve is maximal.
We remark that Cheeger and Colding \cite{CC} proved upper semi-continuity of the eigenvalue functional in a general setting.
In our situation, a simple argument applies.

Recall $R$ is the maximum of the radii of $P$ and $Q$.
Also $D_R$ is a disc in $\R^2$ of radius $R$ and $\lambda_1(D_R)$ is the first Dirichlet eigenvalue of $D_R$.

\begin{Lemma}
\label{Lipexist}
Assume that $\Lambda > \lambda_1(D_R)$.
Then there is a curve $\alpha$ in $\cR$ such that
\[
	\lambda_1(\alpha) = \Lambda
\]
and, for almost every $t$ in $[0,1]$,
\[
	| \alpha'(t) | = L
\]
where $L$ is the length of $\alpha$.
\end{Lemma}

\begin{proof}
Let $\{ \gamma_k \}$ be a sequence in $\cR$ such that
\[
	\lim_{k \to \infty} \lambda_1(\gamma_k) = \Lambda
\]
Let $L_k$ be the length of $\gamma_k$.
Using Lemma \ref{arcparam}, we may assume that, for every $k$ and almost every $t$ in $[0,1]$,
\[
	| \gamma_k'(t) | = L_k
\]
Note that
\[
	\lambda_1(D_R) = \inf \Big\{ \lambda_1(A_{a,R}) : 0<a<R \Big\}
\]
For a proof, we refer to Rauch and Taylor \cite{RT}, who considered a much more general problem.
Now, since $\Lambda > \lambda_1(D)$, there are positive numbers $a<R$ and $b>R$ such that
\[
	\lambda_1(A_{a,R}) < \Lambda
\]
and
\[
	\lambda_1(A_{R,b}) < \Lambda
\]
We may assume that $\lambda_1(\alpha_k) > \lambda_1(A_{a,R})$ and $\lambda_1(\alpha_k) > \lambda_1(A_{R,b})$ for all $k$.
Write $\gamma_k=(F_k,G_k)$.
By Lemma \ref{Fbound}, the image of $F_k$ is contained in $[a,b]$ for all $k$.
Then, by Lemma \ref{Lbound}, the curves $\gamma_k$ have uniformly bounded length.
In particular, there is a compact subset of $\R^2_+$ which contains the image of $\gamma_k$ for all $k$.
By passing to a subsequence, we may assume that the sequence of lengths $\{ L_k \}$ converges to some positive number $\ell$.
Then the curves $\gamma_k$ are uniformly Lipschitz.
By applying the Arzela-Ascoli theorem and passing to a subsequence, we may assume that $\{ \gamma_k \}$ converges uniformly to a curve $\alpha$ in $\cR$.
Moreover, for almost every $t$ in $[0,1]$,
\[
	| \alpha'(t) | \le \ell
\]
Write $\alpha=(F,G)$.
Let $\epsilon >0$.
There is a $w$ in $\Lip_0(0,1)$ such that
\[
	\frac{ \int_0^1 \frac{| w' |^2 F}{|\alpha'|} \,dt}{\int_0^1 | w |^2 F |\alpha'| \,dt} < \lambda_1(\alpha) + \epsilon
\]
Now
\[
	\lim_{k \to \infty} \frac{ \int_0^1 \frac{| w' |^2 F_k}{L_k} \,dt}{\int_0^1 | w |^2 F_k L_k \,dt} = \frac{ \int_0^1 \frac{| w' |^2 F}{\ell} \,dt}{\int_0^1 | w |^2 F \ell \,dt} \le \frac{ \int_0^1 \frac{| w' |^2 F}{|\alpha'|} \,dt}{\int_0^1 | w |^2 F |\alpha'| \,dt}
\]
This implies that
\[
	\Lambda < \lambda_1(\alpha) + \epsilon
\]
Since $\epsilon$ was arbitrary, we obtain $\lambda_1(\alpha) = \Lambda$.
To complete the proof, we apply Lemma \ref{arcparam}.
\end{proof}

In the next section we consider regularity of a maximizing curve.
We conclude this section by applying Lemma \ref{Fbound} to prove Lemma \ref{topodisc}.

\begin{proof}[Proof of Lemma 1.1]
Without loss of generality, assume $P$ has radius $R$.
Let $\Sigma_P$ be the connected component of $\Sigma$ which contains $P$.
Then $\Sigma_P$ is diffeomorphic to a disc.
Let $L$ be the length of the curve $\Sigma_P \cap \R^2_+$.
Let $\gamma:[0,L) \to \Sigma_P \cap \R^2_+$ be a smooth arclength parametrization with $\gamma(0)$ lying on $P$.
Write $\gamma=(F,G)$.
Then
\[
	\lim_{t \to L^-} F(t) = 0
\]
For $d$ in $(0,L)$, define
\[
	\gamma_d = \gamma \Big|_{[0,d]}
\]
Write $\gamma=(F,G)$.
For every $d$ in $(0,L)$, by Lemma \ref{Fbound},
\[
	\lambda_1(\gamma_d) \le \lambda_1(A_{F(d),R})
\]
We remark that
\[
	\lim_{a \to 0^+} \lambda_1(A_{a,R}) = \lambda_1(D_R)
\]
For a proof, we refer again to Rauch and Taylor, \cite{RT}.
Also, for every $d$ in $(0,L)$,
\[
	\lambda_1(\Sigma) \le \lambda_1(\gamma_d)
\]
Therefore $\lambda_1(\Sigma) \le \lambda_1(D_R)$.
\end{proof}

\section{Regularity}

In this section we complete the proof of Lemma \ref{exreg}.
\begin{Lemma}
\label{efex}
Let $\beta$ be a curve in $\cR$.
Assume that there is a constant $c>0$ such that, for almost every $t$ in $[0,1]$,
\[
	| \beta'(t) | \ge c
\]
Write $\beta=(F,G)$.
Then there exists a function $\phi$ in $\Lip_0(0,1)$ which does not vanish in $(0,1)$ and satisfies
\[
	\lambda_1(\beta) = \frac{ \int_0^1 \frac{| \phi' |^2 F}{| \beta' |} \,dt}{\int_0^1 | \phi |^2 | \beta' | F \,dt}
\]
The function $\phi$ is a weak solution of the equation
\[
	-\bigg( \frac{F}{|\beta'|} \phi' \bigg)' = \lambda_1(\beta) F | \beta' | \phi
\]
Furthermore,
\[
	\inf \bigg\{  \frac{ \int_0^1 \frac{| w' |^2 F}{| \beta' |} \,dt}{\int_0^1 | w |^2 | \beta' | F \,dt} : w \in \Lip_0(0,1), \int_0^1 w \phi | \beta' | F \,dt =0 \bigg\} > \lambda_1(\beta)
\]
\end{Lemma}

This follows from a standard argument.
We omit the proof and refer to Gilbarg and Trudinger \cite{GT} for details.
Next we see that if a maximizing curve in $\cR$ with constant speed intersects a sufficiently small circle at two points, then it must stay inside the circle between those points.

\begin{Lemma}
\label{circles}
Let $\alpha$ be a curve in $\cR$ such that
\[
	\lambda_1(\alpha) = \Lambda
\]
Let $L$ be the length of $\alpha$.
Assume that, for almost every $t$ in $[0,1]$,
\[
	| \alpha'(t) | = L
\]
Let $(x_0,y_0)$ be in $\R^2_+$.
Let $r_0$ be a positive number such that $5r_0 \le x_0$.
Define
\[
	D = \bigg\{ (x,y) \in \R^2_+: (x-x_0)^2+(y-y_0)^2 \le r_0^2 \bigg\}
\]
Let $0 \le t_1 < t_2 \le 1$ and assume $\alpha(t_1)$ and $\alpha(t_2)$ lie on the boundary $\del D$.
Assume that, for all $t$ in $[t_1, t_2]$,
\[
	| \alpha(t) - (x_0,y_0) | < 2 r_0
\]
Then $\alpha(t)$ is in $D$ for all $t$ in $[t_1, t_2]$. 
\end{Lemma}

\begin{proof}
Suppose not.
It suffices to consider the case where $\alpha(t)$ lies outside of $D$ for every $t$ in $(t_1,t_2)$.
There are Lipschitz functions $r:[0,1] \to (0, \infty)$ and $\theta:[0,1] \to \R$ such that
\[
	\alpha(t) = \Big( x_0 + r(t) \cos \theta(t), y_0 + r(t) \sin \theta(t) \Big)
\]
Note that $r_0 < r(t) < 2r_0$ for all $t$ in $(t_1, t_2)$.
Define a curve $\beta$ in $\cR$ by
\[
	\beta(t) =
	\begin{cases}
		\alpha(t) & t \in [0,t_1] \cup [t_2,1] \\
		\Big( x_0 + \frac{r_0^2}{r(t)} \cos \theta(t), y_0 + \frac{r_0^2}{r(t)} \sin \theta(t) \Big) & t \in [t_1, t_2] \\
	\end{cases}
\]
Write $\beta=(F_\beta, G_\beta)$.
Then, over $(t_1,t_2)$,
\[
	\frac{F_\beta}{| \beta' |} > \frac{F_\alpha}{| \alpha'|}
\]
Also, over $(t_1,t_2)$,
\[
	F_\beta | \beta'| < F_\alpha | \alpha'|
\]
By Lemma \ref{efex}, there is a function $\phi$ in $\Lip_0(0,1)$ which is non-vanishing over $(0,1)$ and satisfies
\[
	\lambda_1(\beta) = \frac{ \int_0^1 \frac{| \phi' |^2 F_\beta}{| \beta' |} \,dt}{\int_0^1 | \phi |^2 | \beta' | F_\beta \,dt}
\]
Write $\alpha=(F_\alpha, G_\alpha)$.
Then
\[
	\lambda_1(\alpha) \le \frac{ \int_0^1 \frac{| \phi' |^2 F_\alpha}{| \alpha' |} \,dt}{\int_0^1 | \phi |^2 | \alpha' | F_\alpha \,dt} < \frac{ \int_0^1 \frac{| \phi' |^2 F_\beta}{| \beta' |} \,dt}{\int_0^1 | \phi |^2 | \beta' | F_\beta \,dt} = \lambda_1(\beta)
\]
This is a contradiction, because $\lambda_1(\alpha)=\Lambda$.
\end{proof}

\begin{Lemma}
\label{inject}
There is a curve $\alpha$ in $\cR$ which is injective and satisfies
\[
	\lambda_1(\alpha) = \Lambda
\]
Moreover if $L$ is the length of $\alpha$ then for almost every $t$ in $[0,1]$,
\[
	| \alpha'(t) | = L
\]
\end{Lemma}

\begin{proof}
Let $\alpha_0$ be the curve in $\cR$ given by Lemma \ref{Lipexist}.
Define
\[
	c = \max \bigg\{ t \in [0,1] : \alpha_0(t) = p \bigg\}
\]
and
\[
	d = \min \bigg\{ t \in [0,1] : \alpha_0(t) = q \bigg\}
\]
Then define a curve $\alpha$ in $\cR$ by
\[
	\alpha(t) = \alpha_0\Big( c + t (d-c) \Big)
\]
Changing variables shows that $\lambda_1(\alpha) \ge \lambda_1(\alpha_0)$.
Therefore,
\[
	\lambda_1(\alpha) = \Lambda
\]
If $L$ is the length of $\alpha$, then for almost every $t$ in $[0,1]$,
\[
	| \alpha'(t) | = L
\]

Suppose that $\alpha$ is not injective.
Then there are points $s$ and $u$ in $(0,1)$ such that $s<u$ and $\alpha(s)=\alpha(u)$.
Define a function $\eta:[0,1] \to [0,1]$ by
\[
	\eta(t) =
	\begin{cases}
		2st & 0 \le t \le 1/2 \\
		(2-2u)t + (2u-1) & 1/2 < t \le 1 \\
	\end{cases}
\]
The function $\eta$ is not continuous, but $\beta=\alpha \circ \eta$ is a curve in $\cR$.
Write $\beta=(F_\beta, G_\beta)$.
By Lemma \ref{efex}, there is function $\phi$ in $\Lip_0(0,1)$ which is non-vanishing over $(0,1)$ and satisfies
\[
	\lambda_1(\beta) = \frac{\int_0^1 \frac{ |\phi'|^2 F_\beta }{ | \beta' | } \,dt}{ \int_0^1 |\phi|^2 F_\beta |\beta'| \,dt}
\]
Define a function $\ell:[0,1] \to [0,1]$ by
\[
	\ell(t) =
	\begin{cases}
		\frac{t}{2s} & 0 \le t \le s \\
		1/2 & s \le t \le u \\
		\frac{t+1-2u}{2-2u} & u \le t \le 1 \\
	\end{cases}
\]
Note that $\beta \circ \ell(t) = \alpha(t)$ for $t$ in $[0,s] \cup [u,1]$.
Define $v = \phi \circ \ell$.
Then $v$ is in $\Lip_0(0,1)$ and $v$ is non-vanishing over $(0,1)$.
Write $\alpha=(F_\alpha, G_\alpha)$.
Changing variables yields
\[
\begin{split}
	\lambda_1(\alpha) &\le \frac{\int_0^1 \frac{ |v'|^2 F_\alpha }{ | \alpha' | } \,dt}{ \int_0^1 |v|^2 F_\alpha |\alpha'| \,dt} \\
		&< \frac{\int_0^s \frac{ |v'|^2 F_\alpha }{ | \alpha' | } \,dt + \int_u^1 \frac{ |v'|^2 F_\alpha }{ | \alpha' | } \,dt}{ \int_0^s |v|^2 F_\alpha |\alpha'| \,dt + \int_u^1 |v|^2 F_\alpha |\alpha'| \,dt} \\
		& = \frac{\int_0^1 \frac{ |\phi'|^2 F_\beta }{ | \beta' | } \,dt}{ \int_0^1 |\phi|^2 F_\beta |\beta'| \,dt} = \lambda_1(\beta)
\end{split}
\]
This is a contradiction, because $\lambda_1(\alpha) = \Lambda$.
\end{proof}

The following lemma is a preliminary regularity result for an eigenvalue maximizing curve with constant speed.

\begin{Lemma}
\label{abdiff}
Let $\alpha$ be an injective curve in $\cR$ such that
\[
	\lambda_1(\alpha) = \Lambda
\]
Let $L$ be the length of $\alpha$.
Assume that, for almost every $t$ in $[0,1]$,
\[
	| \alpha'(t) | = L
\]
Let $t_0$ be a point in $[0,1]$.
Let $\{ s_k \}$ be a sequence in $[0, t_0]$ converging to $t_0$ and let $\{ u_k \}$ be a sequence in $[t_0,1]$ converging to $t_0$.
Then
\[
	\lim_{k \to \infty} \frac{| \alpha(u_k)-\alpha(s_k)|}{|u_k-s_k|} = L
\]
In particular,
\[
	\lim_{t \to t_0} \frac{| \alpha(t)-\alpha(t_0)|}{|t-t_0|} = L
\]
\end{Lemma}

\begin{proof}
Suppose not.
Since $\alpha$ is Lipschitz with constant $L$,
\[
	\limsup_{k \to \infty} \frac{| \alpha(u_k) - \alpha(s_k)|}{|u_k - s_k|} \le L
\]
Therefore,
\[
	\liminf_{k \to \infty} \frac{| \alpha(u_k)-\alpha(s_k)|}{|u_k - s_k|} < L
\]
Let $c$ be a constant such that
\[
	\liminf_{k \to \infty} \frac{| \alpha(u_k)-\alpha(s_k)|}{|u_k-s_k|} < c < L
\]
By passing to subsequences, we may assume that for all $k$,
\[
	\frac{| \alpha(u_k)-\alpha(s_k)|}{|u_k-s_k|} < c
\]
Fix $k$ large, and define a curve $\beta$ in $\cR$ by
\[
	\beta(t) =
	\begin{cases}
		\alpha(t) & 0\le t \le s_k \\
		\alpha(s_k) + (t-s_k) \frac{\alpha(u_k)-\alpha(s_k)}{u_k-s_k} & s_k \le t \le u_k\\
		\alpha(t) & u_k \le t \le 1
	\end{cases}
\]
Write $\beta=(F_\beta,G_\beta)$.
Since $\alpha$ is injective, Lemma \ref{efex} shows there is a function $\phi$ in $\Lip_0(0,1)$ which is non-vanishing over $(0,1)$ and satisfies
\[
	\lambda_1(\beta) = \frac{ \int_0^1 \frac{| \phi' |^2 F_\beta}{| \beta' |} \,dt}{\int_0^1 | \phi |^2 | \beta' | F_\beta \,dt}
\]
Write $\alpha=(F_\alpha, G_\alpha)$.
If $k$ is sufficiently large, then over $(s_k, u_k)$,
\[
	\frac{ F_\beta(t)}{| \beta'(t)|} > \frac{ F_\alpha(t)}{| \alpha'(t)|}
\]
and
\[
	 F_\beta(t) | \beta'(t) | < F_\alpha(t) | \alpha'(t)|
\]
So if $k$ is sufficiently large,
\[
	\lambda_1(\alpha) \le \frac{ \int_0^1 \frac{| \phi' |^2 F_\alpha}{| \alpha' |} \,dt}{\int_0^1 | \phi |^2 | \alpha' | F_\alpha \,dt} < \frac{ \int_0^1 \frac{| \phi' |^2 F_\beta}{| \beta' |} \,dt}{\int_0^1 | \phi |^2 | \beta' | F_\beta \,dt} = \lambda_1(\beta)
\]
This is a contradiction, because $\lambda_1(\alpha) = \Lambda$.
\end{proof}

Now we can show that if an eigenvalue maximizing curve has constant speed, then it is differentiable.

\begin{Lemma}
\label{diff}
Let $\alpha$ be an injective curve in $\cR$ such that
\[
	\lambda_1(\alpha) = \Lambda
\]
Let $L$ be the length of $\alpha$.
Assume that, for almost every $t$ in $[0,1]$,
\[
	| \alpha'(t) | = L
\]
Then $\alpha$ is differentiable over $[0,1]$.
Moreover, for every $t$ in $[0,1]$,
\[
	|\alpha'(t)|= L
\]
\end{Lemma}

\begin{proof}
We first prove that $\alpha$ is right-differentiable over $[0,1)$.
Let $t_0$ be in $[0,1)$ and suppose that $\alpha$ is not right-differentiable at $t_0$.
It follows from Lemma \ref{abdiff}, that there is a positive constant $c$ and sequences $\{ y_k \}$ and $\{ z_k \}$ in $(t_0,1]$ converging to $t_0$ such that, for all $k$, the points $\alpha(t_0), \alpha(y_k), \alpha(z_k)$ are distinct, and the interior angle at $\alpha(t_0)$ of the triangle with vertices at these points is at least $c$.
By passing to a subsequence we may assume that $y_k < z_k$ for all $k$.
Write $\alpha=(F,G)$ and fix a positive constant $r_0$ with
\[
	6 r_0 < F(t_0)
\]
For large $k$,
\[
	0 < | \alpha(z_k) - \alpha(t_0) | < r_0
\]
Then there are two closed discs of radius $r_0$ which contain $\alpha(z_k)$ and $\alpha(t_0)$ on their boundaries.
If $k$ is large, then by Lemma \ref{circles}, the point $\alpha(y_k)$ must be in the intersection of these discs.
But this implies that the interior angle at $\alpha(t_0)$ of the triangle with vertices at $\alpha(t_0), \alpha(y_k), \alpha(z_k)$ converges to zero as $k \to \infty$.
By this contradiction $\alpha$ is right-differentiable over $[0,1)$.

A symmetric argument shows that $\alpha$ is left-differentiable over $(0,1]$.
Then Lemma \ref{abdiff} implies that the left and right derivatives must agree over $(0,1)$.
Lemma \ref{abdiff} also implies that for every $t$ in $[0,1]$,
\[
	| \alpha'(t) | = L
\]
\end{proof}

The following lemma shows that if a maximizing curve has constant speed, then it is continuously differentiable.

\begin{Lemma}
\label{c1}
Let $\alpha$ be an injective curve in $\cR$ such that
\[
	\lambda_1(\alpha) = \Lambda
\]
Let $L$ be the length of $\alpha$.
Assume that $\alpha$ is differentiable over $[0,1]$ and, for every $t$ in $[0,1]$,
\[
	| \alpha'(t) | = L
\]
Then $\alpha$ is continuously differentiable over $[0,1]$.
\end{Lemma}

\begin{proof}
Fix $t_0$ in $[0,1]$ and let $\{ s_k \}$ be a sequence in $[0,1]$ converging to $t_0$.
Write $\alpha=(F, G)$ and let $r_0>0$ be such that
\[
	6 r_0 < F(t_0)
\]
For large $k$, there are exactly two closed discs in $\R^2_+$ of radius $r_0$ which contain $\alpha(s_k)$ and $\alpha(t_0)$ on their boundaries.
If $k$ is large, then Lemma \ref{circles} implies that $\alpha(t)$ must lie in the intersection of these discs for all $t$ between $t_0$ and $s_k$.
By assumption, $\alpha$ is differentiable over $[0,1]$, and for all $t$ in $[0,1]$,
\[
	| \alpha'(t) | = L
\]
It follows that
\[
	\lim_{k \to \infty} | \alpha'(s_k) - \alpha'(t_0) | = 0
\]
Therefore $\alpha'$ is continuous at $t_0$.
\end{proof}

%\begin{Lemma}
%\label{c11}
%Let $\alpha$ be a curve in $\cR$ such that
%\[
%	\lambda(\alpha) = \Lambda
%\]
%Let $L$ be the length of $\alpha$.
%Assume that $\alpha$ is differentiable over $[0,1]$ and, for every $t$ in $[0,1]$,
%\[
%	| \alpha'(t) | = L
%\]
%Then $\alpha'$ is Lipschitz continuous over $[0,1]$.
%\end{Lemma}
%
%\begin{proof}
%Suppose not.
%Then there are sequences $\{ p_k \}$ and $\{ q_k \}$ in $[0,1]$, such that $p_k < q_k$ for all $k$, and
%\[
%	\lim_{k \to \infty} \frac{ | \alpha'(p_k) - \alpha'(q_k) | }{|p_k - q_k|} = \infty
%\]
%By passing to subsequences, we may assume that $\{ p_k \}$ converges to some $t_0$ in $[0,1]$.
%It necessarily follows that $\{ q_k \}$ converges to $t_0$, as well.
%Write $\alpha=(F_\alpha, F_\beta)$ and let $r_0>0$ be such that, for all $t$ in $[0,1]$,
%\[
%	F_\alpha(t) \ge 5r_0
%\]
%By passing to a subsequence we may assume that the distance between $\alpha(p_k)$ and $\alpha(q_k)$ is less than $r_0$ for all $k$.
%Then there are exactly two closed discs in $\R^2_+$ of radius $r_0$ which contain $\alpha(p_k)$ and $\alpha(q_k)$ on their boundaries.
%If $k$ is sufficiently large, then Lemma \ref{circles} implies that $\alpha(t)$ must lie in the intersection of the discs for all $t$ in $(p_k,q_k)$.
%By assumption, $\alpha$ is differentiable over $[0,1]$, and for all $t$ in $[0,1]$,
%\[
%	| \alpha'(t) | = L
%\]
%It follows that
%\[
%	\frac{ | \alpha'(p_k) - \alpha'(q_k) | }{|p_k - q_k|} \le \frac{2L^2}{r_0}
%\]
%This is a contradiction.
%\end{proof}

In the following lemma, we establish differentiability of the first eigenvalue functional along a variation of a continuously differentiable curve.

\begin{Lemma}
\label{lambdadiff}
Let $\alpha$ be a curve in $\cR$ such that
\[
	\lambda_1(\alpha)=\Lambda
\]
Let $L$ be the length of $\alpha$.
Assume that $\alpha$ is continuously differentiable over $[0,1]$ and, for every $t$ in $[0,1]$,
\[
	| \alpha'(t) | = L
\]
Let $v$ be a unit vector in $\R^2$.
Let $\psi: [0,1] \to \R$ be a smooth function with $\psi(0)=\psi(1)=0$.
For small $s$, let $\alpha_s$ be the curve in $\cR$ defined by
\[
	\alpha_s(t)=\alpha(t)+s\psi(t)v
\]
Then the function $s \mapsto \lambda_1(\alpha_s)$ is differentiable at zero.
\end{Lemma}

\begin{proof}
For small $s$, write $\alpha_s=(F_s,G_s)$.
By Lemma \ref{efex}, there are functions $\phi_s$ in $\Lip_0(0,1)$ such that
\[
	\lambda_1(\alpha_s) = \frac{\int_0^1 \frac{ F_s |\phi_s'|^2}{| \alpha_s' |} \,dt } { \int_0^1 F_s | \alpha_s' | \phi_s^2 \,dt}
\]
We may assume that
\[
	\int_0^1 \phi_s^2 F_s | \alpha_s'| \, dt = 1
\]
Let $H_0^1$ be the Sobolev space of continuous functions $w:[0,1] \to \R$ which vanish at the endpoints and have a weak derivative in $L^2(0,1)$.
Let $H^{-1}$ be the dual space of $H_0^1$.
Define operators $\L_s:H_0^1 \to H^{-1}$ by
\[
	(\L_s u , v) = \int_0^1 \frac{F_s u'v'}{|\alpha'_s|} \,dt
\]
Also define operators $i_s:L_2 \to H^{-1}$ by
\[
	(i_s u, v) = \int_0^1 uv F_s | \alpha'_s| \,dt
\]
Then
\[
	\L_s \phi_s = \lambda_1(\alpha_s) i_s \phi_s
\]
Let $\epsilon>0$ be small and define $\Psi: \R \times (-\epsilon, \epsilon) \times H_0^1 \to H^{-1} \times \R$ by
\[
	\Psi(\lambda,s,u) = \bigg( \L_s u - \lambda i_s u , \int_0^1 u^2 F_s |\alpha_s'| \,dt \bigg)
\]
This map is continuously differentiable.
Moreover,
\[
	\Psi(\lambda_1(\alpha_s), s, \phi_s) = (0, 1)
\]
Note that
\[
	0 \le \lambda_1(\alpha_0) - \lambda_1(\alpha_s) \le \frac{\int_0^1 \frac{ F_0 |\phi_s'|^2}{| \alpha_0' |} \,dt } { \int_0^1 F_0 | \alpha_0' | \phi_s^2 \,dt} - \frac{\int_0^1 \frac{ F_s |\phi_s'|^2}{| \alpha_s' |} \,dt } { \int_0^1 F_s | \alpha_s' | \phi_s^2 \,dt}
\]
Moreover
\[
	\lim_{s \to 0} \bigg( \frac{\int_0^1 \frac{ F_0 |\phi_s'|^2}{| \alpha_0' |} \,dt } { \int_0^1 F_0 | \alpha_0' | \phi_s^2 \,dt} - \frac{\int_0^1 \frac{ F_s |\phi_s'|^2}{| \alpha_s' |} \,dt } { \int_0^1 F_s | \alpha_s' | \phi_s^2 \,dt} \bigg) = 0
\]
Therefore
\[
	\lim_{s \to 0} \lambda_1(\alpha_s) = \lambda_1(\alpha_0)
\]
It then follows from the last statement in Lemma \ref{efex} that
\[
	\lim_{s \to 0} \phi_s = \phi_0
\]
with convergence in $H_0^1$.
That is, the curve in $\R \times H_0^1$ defined by
\[
	s \mapsto \Big( \lambda_1(\alpha_s), \phi_s \Big)
\]
is continuous at $s=0$.
It then follows from the implicit function theorem that this curve is differentiable at $s=0$.
\end{proof}

\begin{Lemma}
\label{ode}
Let $\alpha$ be a curve in $\cR$ and assume that $\lambda_1(\alpha)=\Lambda$.
Let $L$ be the length of $\alpha$.
Assume that $\alpha$ is continuously differentiable over $[0,1]$, and for every $t$ in $[0,1]$,
\[
	| \alpha'(t) | = L
\]
Write $\alpha=(F,G)$.
Let $\phi$ be in $\Lip_0(0,1)$ and assume that
\[
	\lambda_1(\alpha) = \frac{ \int_0^1 \frac{ F | \phi' |^2 } {L} \,dt }{ \int_0^1 F \phi^2 L \,dt }
\]
Then the following equations hold weakly over $(0,1)$,
\[
\begin{cases}
	\Big( ( |\phi'|^2 + \Lambda L^2 \phi^2) FF' \Big) ' = \Lambda L^4 \phi^2 - L^2 |\phi'|^2 \\
	\Big( (| \phi' |^2 + \Lambda L^2 \phi^2 ) FG' \Big) ' = 0
\end{cases}
\]
\end{Lemma}

\begin{proof}
To establish the first equation, let $\psi:[0,1] \to \R$ be a smooth function with $\psi(0)=\psi(1)=0$.
For small $s$, define $F_s=F+s\psi$, and let $\alpha_s=(F_s, G)$.
Then each $\alpha_s$ is a curve in $\cR$, and the function $s \mapsto \lambda_1(\alpha_s)$ is differentiable at zero by Lemma \ref{lambdadiff}.
The derivative at zero must vanish.
For small $s$, define
\[
	R(s) = \frac{ \int_0^1 \frac{| \phi' |^2 F_s}{| \alpha'_s|} \,dt}{\int_0^1  \phi^2 F_s | \alpha'_s| \,dt}
\]
Then $\lambda_1(\alpha)=R(0)$ and $\lambda_1(\alpha_s) \le R(s)$ for small $s$.
The function $R$ is differentiable at zero, so its derivative at zero must vanish, as well.
That is,
\[
	\int_0^1 L^2 | \phi' |^2 \psi - | \phi' |^2 F F' \psi' \,dt = \int_0^1 \Lambda L^4 \phi^2 \psi + \Lambda L^2 \phi^2 F F' \psi' \,dt
\]
This is the weak formulation of the equation
\[
	\Big((|\phi'|^2 + \Lambda L^2 \phi^2) F F'\Big)' = \Lambda L^4 \phi^2 - L^2 |\phi'|^2
\]

To establish the second equation, let $\eta:[0,1] \to \R$ be a smooth function with $\psi(0)=\psi(1)=0$.
For small $s$, define $G_s=G+s\eta$ and let $\alpha_s=(F, G_s)$.
Then each $\alpha_s$ is a curve in $\cR$, and the function $s \mapsto \lambda_1(\alpha_s)$ is differentiable at zero by Lemma \ref{lambdadiff}.
The derivative at zero must vanish.
For small $s$, define
\[
	Q(s) = \frac{ \int_0^1 \frac{| \phi' |^2 F}{| \alpha'_s|} \,dt}{\int_0^1  \phi^2 F | \alpha'_s| \,dt}
\]
Then $\lambda_1(\alpha)=Q(0)$ and $\lambda_1(\alpha_s) \le Q(s)$ for small $s$.
The function $Q$ is differentiable at zero, so its derivative at zero must vanish, as well.
That is,
\[
	\int_0^1 | \phi' |^2 F G' \eta' + \Lambda L^2 \phi^2 F G' \eta' \,dt = 0
\]
This is the weak formulation of the equation
\[
	((|\phi'|^2 + \Lambda L^2 \phi^2) F G'\Big)' = 0
\]
\end{proof}

\begin{Lemma}
\label{smooth}
Let $\alpha$ be a curve in $\cR$ and assume that $\lambda_1(\alpha)=\Lambda$.
Let $L$ be the length of $\alpha$.
Assume that $\alpha$ is continuously differentiable over $[0,1]$ and, for every $t$ in $[0,1]$,
\[
	| \alpha'(t) | = L
\]
Then $\alpha$ is smooth.
\end{Lemma}

\begin{proof}
Write $\alpha=(F,G)$.
By Lemma \ref{efex}, there is a function $\phi$ in $\Lip_0(0,1)$ which is non-vanishing over $(0,1)$ and satisfies
\[
	\lambda_1(\alpha) = \frac{ \int_0^1 \frac{ F | \phi' |^2 } {L} \,dt }{ \int_0^1 F \phi^2 L \,dt }
\]
Moreover, Lemma \ref{efex} states that, weakly over $(0,1)$,
\[
	-\bigg( \frac{F \phi'}{L} \bigg)' = \Lambda L \phi F
\]
Now the functions $F$ and $G$ are $C^1$, so $\phi$ is $C^2$ and the equation holds strongly.
In particular, $\phi'$ does not vanish at $0$ or $1$.
We complete the argument by induction.
Let $k=1,2,3,\ldots$ and assume that $F$ and $G$ are $C^k$ and $\phi$ is $C^{k+1}$.
Then
\[
	(| \phi' |^2 + \Lambda L^2 \phi^2 )^{-1}
\]
is $C^k$.
Lemma \ref{ode} implies that $F$ and $G$ are $C^{k+1}$.
This in turn implies that $\phi$ is $C^{k+2}$.
By induction, the functions $F$, $G$, and $\phi$ are all smooth.
\end{proof}

We can now prove Lemma \ref{exreg}.

\begin{proof}[Proof of Lemma 1.2]
By Lemma \ref{inject}, there is an injective curve $\alpha$ in $\cR$ of length $L$ such that
\[
	\lambda_1(\alpha) = \Lambda
\]
and for almost every $t$ in $[0,1]$,
\[
	| \alpha'(t) | = L
\]
By Lemma \ref{diff}, the curve $\alpha$ is differentiable over $[0,1]$ and, for every $t$ in $[0,1]$,
\[
	| \alpha'(t) | = L
\]
By Lemma \ref{c1}, $\alpha$ is continuously differentiable over $[0,1]$.
By Lemma \ref{smooth}, the curve $\alpha$ is smooth.
Therefore,
\[
	\lambda_1(\alpha) = \sup \{ \lambda_1(\gamma) : \gamma \in \cC \}
\]
\end{proof}

\section{Uniqueness}

In this section we prove Lemma \ref{unique}.
We first show that maximizing curves yield solutions to an initial value problem.

\begin{Lemma}
\label{maxcrit}
Let $\alpha$ be a curve in $\cC$ such that
\[
	\lambda_1(\alpha) = \Lambda
\]
Let $L$ be the length of $\alpha$ and assume that for all $t$ in $[0,1]$,
\[
	| \alpha'(t) | = L
\]
Define $\beta:[0,L] \to \R^2_+$ by
\[
	\beta(t) = \alpha(t/L)
\]
Let $\Theta$ be in $\R$ be such that
\[
	\beta'(0) = ( \cos \Theta, \sin \Theta)
\]
Write $\beta=(F,G)$.
Then there are smooth functions $v:[0,L] \to \R$ and $\theta:[0,L] \to \R$ with $v$ positive over $(0,L)$ and vanishing at $L$ such that $v$, $\theta$, $F$, and $G$ satisfy
\[
\begin{cases}
	v'' = - ( \frac{ \cos \theta}{F} ) v' - \Lambda v \\
	\theta' = \frac{\sin \theta ( | v'| ^2 - \Lambda v^2 )}{ F ( | v' |^2 + \Lambda v^2 ) } \\
	F' = \cos \theta \\
	G' = \sin \theta \\
	v(0) = 0, \quad v'(0) = 1 \\
	\theta(0) = \Theta \\
	F(0) = p_1, \quad G(0) = p_2 \\
\end{cases}
\]
The equations hold over $(0,L)$.
Moreover, the function $v$ is non-vanishing over $(0,L)$ and satisfies $v(L)=0$.
\end{Lemma}

\begin{proof}
There is a smooth function $\theta: [0,L] \to \R$ such that $\theta(0)=\Theta$ and for all $t$ in $[0,L]$,
\[
	\beta'(t) = \Big( \cos \theta(t), \sin \theta(t) \Big)
\]
Let $\Sigma$ be the surface of revolution obtained from $\beta$.
Let $\phi: \Sigma \to \R$ be a first Dirichlet eigenfunction.
Since $\Sigma$ is a surface of revolution, the first Dirichlet eigenfunction $\phi$ is rotationally symmetric.
That is, it can be identified with a smooth function $v:[0,L] \to \R$, which is non-vanishing over $(0,L)$ and satisfies
\[
	v'' = - \bigg( \frac{\cos \theta}{F} \bigg) v' - \Lambda v
\]
Additionally, we have $v(0)=0$ and $v(L)=0$.
We may also assume that $v'(0)=1$.
Then the immersion $\Sigma \to \R^3$ must be critical with respect to normal, linear variations that preserve the rotational symmetry.
It follows from Berger's variational formula \cite{B} that
\[
	(| \grad \phi |^2 - \Lambda \phi^2 ) H = 2\II(\grad \phi, \grad \phi)
\]
Here $\grad$ is the gradient defined with respect to the pullback metric, and $\II$ and $H$ are the second fundamental form and mean curvature.
This equation can be rewritten as
\[
	(|v'|^2 - \Lambda v^2) \bigg( \theta' + \frac{ \sin \theta}{F} \bigg) = 2 | v' |^2 \theta'
\]
\end{proof}

Let $D$ be a disc in $\R^2$ of radius $p_1$.
Let $\lambda_1(D)$ be the first Dirichlet eigenvalue of $D$.

\begin{Lemma}
\label{critex}
Let $\Theta$ be in $\R$ and let $\lambda>\lambda_1(D)$.
Let $v$, $\theta$, $F$, and $G$ solve the initial value problem
\[
\begin{cases}
	v'' = - ( \frac{ \cos \theta}{F} ) v' - \lambda v \\
	\theta' = \frac{\sin \theta ( | v'| ^2 - \lambda v^2 )}{ F ( | v' |^2 + \lambda v^2 ) } \\
	F' = \cos \theta \\
	G' = \sin \theta \\
	v(0) = 0, \quad v'(0) = 1 \\
	\theta(0) = \Theta \\
	F(0) = p_1, \quad G(0) = p_2 \\
\end{cases}
\]
Then there is a positive $L$ in the maximal interval of existence such that $v$ is positive over $(0,L)$ with $v(L)=0$.
\end{Lemma}

We will prove this in the next section.
For now we assume the lemma and use it complete the uniqueness part of the theorem.
Let $B>\lambda_1(D)$ and define
\[
	U = \{ (x,y) \in \R^2: x^2+y^2 < 1/B \}
\]
Define a map $\Phi: U \to \R^2_+$ as follows.
First define $\Phi(0,0)=p$.
Now let $(x,y)$ be a point in $U$, with $(x,y) \neq (0,0)$.
There is a $\lambda > B$ and a $\Theta$ in $\R$ such that
\[
	(x,y)=\lambda^{-1/2} (\cos \Theta, \sin \Theta)
\]
Let $v$, $\theta$, $F$, $G$, and $L$ be as in Lemma \ref{critex}.
Now define
\[
	\Phi(x,y) = \Big( F(L), G(L) \Big)
\]

\begin{Lemma}
\label{diffeo}
If $B$ is sufficiently large, then $\Phi$ is a $C^1$ diffeomorphism of $U$ onto a neighborhood of $p$ in $\R^2_+$.
\end{Lemma}

\begin{proof}
Define
\[
	V = [0, B^{-1/2}) \times \R
\]
Define $\Psi:V \to \R^2$ by
\[
	\Psi(\sigma, \Theta)=
	\begin{cases}
		\frac{\Phi(\sigma \cos \Theta, \sigma \sin \Theta) - p}{\sigma} & \sigma \neq 0 \\
		(\pi \cos \Theta, \pi \sin \Theta) & \sigma = 0 \\
	\end{cases}
\]
We can realize $\Psi$ in an alternate way.
Let $(\sigma,\Theta)$ be in $V$.
Let $v_0$, $\theta_0$, $F_0$, and $G_0$ solve the initial value problem:
\[
\begin{cases}
	v_0'' = - \sigma \Big( \frac{ \cos \theta_0}{p_1 + \sigma F_0} \Big) v_0' - v_0 \\
	\theta_0' = \frac{\sigma \sin \theta_0 ( | v_0'| ^2 - v_0^2 )}{ (p_1 + \sigma F_0) ( | v_0' |^2 + v_0^2 ) } \\
	F_0' = \cos \theta_0 \\
	G_0' = \sin \theta_0 \\
	v_0(0) = 0, \quad v_0'(0) = 1 \\
	\theta_0(0) = \Theta \\
	F_0(0) = 0, \quad G_0(0) = 0 \\
\end{cases}
\]
We first claim that there is a positive $L_0$ in the maximal interval of existence such that $v_0$ is positive over $(0, L_0)$ and $v_0(L_0) =0$.
This is trivial if $\sigma = 0$.
If $\sigma > 0$, then this follows from Lemma \ref{critex} by setting $\lambda = 1/\sigma^2$ and defining functions $v$, $\theta$, $F$, and $G$ by
\[
\begin{split}
	&v(t) = \sigma v_0(t/\sigma) \\
	&\theta(t) = \theta_0(t/\sigma) \\
	&F(t) = p_1 + \sigma F_0(t/\sigma) \\
	&G(t) = p_2 + \sigma G_0(t/\sigma)
\end{split}
\]
We now observe that
\[
	\Psi(\sigma,\Theta) = \Big( F_0(L_0), G_0(L_0) \Big)
\]
It follows from the implicit function theorem that $\Psi$ is smooth.
Moreover, $\Phi$ is continuously differentiable and the differential of $\Phi$ at $(0,0)$ is $\pi$ times the identity map, because
\[
	\frac{| \Phi( \sigma \cos \Theta, \sigma \sin \Theta) - p - (\pi \sigma \cos \Theta, \pi \sigma \sin \Theta) |}{\sigma} = |  \Psi(\sigma, \Theta) - \Psi(0, \Theta) |
\]
This converges to zero as $\sigma$ tends to zero, uniformly in $\Theta$.
Now the inverse function theorem yields the lemma.
\end{proof}

\begin{proof}[Proof of Lemma 1.3]
Let $B$ be sufficiently large so that, by Lemma \ref{diffeo}, the function $\Phi$ is a $C^1$ diffeomorphism of $U$ onto a neighborhood of $p$ in $\R^2_+$.
Assume $q$ is sufficiently close to $p$ so that
\[
	\Lambda > B
\]
Suppose that $\alpha_1$ and $\alpha_2$ are curves in $\cC$ such that
\[
	\lambda_1(\alpha_1) = \lambda_1(\alpha_2) = \Lambda
\]
There are $\Theta_1$ and $\Theta_2$ in $[0,2\pi)$ such that
\[
	\frac{\alpha'_1(0)}{|\alpha'_1(0)|} = ( \cos \Theta_1, \sin \Theta_1)
\]
and
\[
	\frac{\alpha'_2(0)}{|\alpha'_2(0)|} = ( \cos \Theta_2, \sin \Theta_2)
\]
It follows from Lemma \ref{maxcrit} that
\[
	\Phi( \Lambda^{-1/2} \cos \Theta_1, \Lambda^{-1/2} \sin \Theta_1) = \Phi (\Lambda^{-1/2} \cos \Theta_2, \Lambda^{-1/2} \sin \Theta_2) = q
\]
Then $\Theta_1$ and $\Theta_2$ differ by an integer multiple of $2\pi$.
Now by Lemma \ref{maxcrit}, the curves $\alpha_1$ and $\alpha_2$ are reparametrizations of each other.
\end{proof}

\section{Critical Surfaces of Revolution}

In this section we prove Lemma \ref{critex}.
Let $\lambda$, $\Theta$, $p_1$, and $p_2$ be real numbers with $p_1>0$.
Let $v$, $\theta$, $F$, and $G$ solve the initial value problem
\[
\begin{cases}
	v'' = - ( \frac{ \cos \theta}{F} ) v' - \lambda v \\
	\theta' = \frac{\sin \theta ( | v'| ^2 - \lambda v^2 )}{ F ( | v' |^2 + \lambda v^2 ) } \\
	F' = \cos \theta \\
	G' = \sin \theta \\
	v(0) = 0, \quad v'(0) = 1 \\
	\theta(0) = \Theta \\
	F(0) = p_1, \quad G(0) = p_2 \\
\end{cases}
\]
Let $[0,\mu)$ be an interval over which the solutions exist.
For $h$ in $(0,\mu)$, define
\[
	\Sigma_h = \bigg\{ (F(t) \cos \omega, F(t) \sin \omega, G(t)) : t \in [0,h], \omega \in \R \bigg\}
\]
Note $F$ is positive over $(0,\mu)$, so $\Sigma_h$ is an immersed hypersurface in $\R^3$.
Let $\lambda_1(\Sigma_h)$ be the first Dirichlet eigenvalue of $\Sigma_h$.

\begin{Lemma}
\label{critjacobi}
Let $h$ be in $(0,\mu)$.
If $v$ is positive over $(0,h)$, then
\[
	\lambda \le \lambda_1(\Sigma_h)
\]
\end{Lemma}

\begin{proof}
Define $\phi: \Sigma_h \to \R$ by
\[
	\phi(F(t) \cos \omega, F(t) \sin \omega, G(t)) = v(t)
\]
Let $\Delta$ be the Laplace-Beltrami operator on $\Sigma_h$.
Then
\[
	-\Delta \phi = \lambda \phi
\]
Let $\Phi$ be the positive first Dirichlet eigenfunction of $\Sigma_h$.
Let $\nu$ be the outward-pointing unit normal vector on $\del \Sigma$.
Then
\[
	\int_{\Sigma_h} (\lambda-\lambda_1(\Sigma_h)) \phi \Phi \,dS = \int_\Sigma \phi \Delta \Phi - \Phi \Delta \phi \,dS = \int_{\del \Sigma} \phi \frac{\del \Phi}{\del \nu} \,ds \le 0
\]
Therefore, $\lambda \le \lambda_1(\Sigma_h)$.
\end{proof}

For $R > r > 0$, let $A_{r,R}$ be the concentric annulus in $\R^2$ with outer radius $R$ and inner radius $r$, centered at the origin.
Let $D$ be a disc in $\R^2$ of radius $p_1$.
Let $\lambda_1(A_{r,R})$ and $\lambda_1(D)$ be the first Dirichlet eigenvalues of $A_{r,R}$ and $D$, respectively.
Note that
\[
	\lambda_1(D) = \inf \bigg\{ \lambda_1 (A_{r,p_1}) : 0 < r < p_1 \bigg\}
\]
For a proof of this, we again refer to Rauch and Taylor \cite{RT}.
Define
\[
	b=\sup \bigg\{ F(t) : t \in [0,\mu) \bigg\}
\]
and
\[
	a=\inf \bigg\{ F(t) : t \in [0,\mu) \bigg\}
\]

\begin{Lemma}
\label{critfbound}
Assume $\lambda > \lambda_1(D)$ and $v$ is positive over $(0,\mu)$.
Then $a$ is positive, $b$ is finite, and
\[
	\lambda \le \lambda_1(A_{a,b})
\]
\end{Lemma}

\begin{proof}
Let $h$ be in $(0,\mu)$.
Let
\[
	R=\max \bigg\{ F(t) : t \in [0,h] \bigg\}
\]
and
\[
	r=\min \bigg\{ F(t): t \in [0,h] \bigg\}
\]
Let $[t_1,t_2] \subset [0,h]$ be such that
\[
	\Big\{ F(t_1), F(t_2) \Big\} = \Big\{ r,R \Big\}
\]
Define
\[
	\Sigma_* = \bigg\{ (F(t) \cos \omega, F(t) \sin \omega, G(t)) : t \in [t_1,t_2], \omega \in \R \bigg\}
\]
Then $\lambda_1(\Sigma_h) \le \lambda_1(\Sigma_*)$.
Let $\gamma:[t_1,t_2] \to \R^2_+$ be the smooth regular curve defined by $\gamma=(F,G)$.
Then
\[
	\lambda_1(\Sigma_*) = \lambda_1(\gamma)
\]
By Lemmas \ref{Fbound} and \ref{critjacobi},
\[
	\lambda \le \lambda_1(A_{r,R})
\]
Since $\lambda > \lambda_1(D)$ this implies that $R$ is bounded above and $r$ is bounded below, independent of $h$.
Therefore $a$ is positive and $b$ is finite.
\end{proof}

\begin{Lemma}
\label{critLbound}
Assume that $v$ is positive over $(0,\mu)$.
Assume $a$ is positive and $b$ is finite.
Then $\mu$ is finite, and
\[
	\mu \le \pi \sqrt{\frac{b}{a \lambda}}
\]
\end{Lemma}

\begin{proof}
Let $h$ be a real number in $(0,\mu)$.
By Lemma \ref{Lbound} and Lemma \ref{critjacobi},
\[
	h \le \pi \sqrt{\frac{b}{a \lambda}}
\]
\end{proof}

\begin{Lemma}
\label{critvbound}
Assume that $a$ is positive and $b$ is finite.
Also assume that $v$ is positive over $(0,\mu)$.
Then $v$ and $v'$ are bounded over $[0,\mu)$.
\end{Lemma}

\begin{proof}
By Lemma \ref{critLbound}, $\mu$ is finite.
Let $h$ be in $(0,\mu)$.
Define $\phi: \Sigma_h \to \R$ by
\[
	\phi(F(t) \cos \omega, F(t) \sin \omega, G(t)) = v(t)
\]
Let $\Delta$ be the Laplace-Beltrami operator on $\Sigma_h$.
Then
\[
	-\Delta \phi = \lambda \phi
\]
If $\nu$ is the outward-pointing unit normal vector on $\del \Sigma$, then, by Green's identity,
\[
	\lambda \int_{\Sigma_h} \phi \,dS = \int_{\Sigma_h} -\Delta \phi \,dS = \int_{\del \Sigma_h} - \frac{\del \phi}{\del \nu} \,ds = 2\pi p_1- 2\pi F(h)v'(h)
\]
In particular,
\[
	v'(h) F(h) \le p_1
\]
This proves that $v'$ is bounded above over $[0,\mu)$.
Now $v$ is bounded above over $(0,\mu)$, because $\mu$ is finite by Lemma \ref{critLbound}.
It follows that there is a $C>0$, independent of $h$ in $(0,\mu)$, such that
\[
	\lambda \int_{\Sigma_h} \phi < C
\]
This implies that $v'$ is bounded below.
\end{proof}

\begin{Lemma}
\label{critvsumbound}
Assume $a$ is positive, $b$ is finite, and $\mu$ is finite.
Assume $v$ is positive over $(0,\mu)$.
Then there is a positive constant $c$ such that
\[
	\lambda |v(t)|^2 + | v'(t)|^2 \ge c
\]
for all $t$ in $(0,\mu)$.
\end{Lemma}

\begin{proof}
Define
\[
	\Sigma = \bigg\{ \Big( (F(t) \cos \omega, F(t) \sin \omega, G(t) \Big) : t \in [0,\mu), \omega \in \R \bigg\}
\]
Note that there is a function $\rho:[0,\mu) \to \R$ which satisfies
\[
	\begin{cases}
		\rho'' = -\frac{\cos \theta}{f} \rho' \\
		\rho(0) = 0 \\
		\rho'(0) = 1
	\end{cases}
\]
Define $\zeta: \Sigma \to \R$ by
\[
	\zeta \Big( F(t) \cos \omega, F(t) \sin \omega, G(t) \Big) = \rho(t)
\]
Then $\zeta$ is positive, harmonic, bounded, and has bounded gradient.
Let $h$ be in $(0,\mu)$.
If $\nu$ is the outward-pointing unit normal vector on $\del \Sigma_h$, then by Green's identity,
\[
	\int_{\Sigma_h} \lambda \phi \zeta \,dS = \int_{\Sigma_h} \phi \Delta \zeta - \zeta \Delta \phi \,dS = \int_{\del \Sigma_h} \phi \frac{\del \zeta}{\del \nu} - \zeta \frac{\del \phi}{\del \nu} \,ds
\]
Therefore,
\[
	\liminf_{h \to \mu} \int_{\del \Sigma_h} \phi \frac{\del \zeta}{\del \nu} - \zeta \frac{\del \phi}{\del \nu} \,ds > 0
\]
It follows that
\[
	\liminf_{h \to \mu} |v(h)|^2 + |v'(h)|^2 > 0
\]
Since $v$ and $v'$ cannot vanish simultaneously in $[0,\mu)$, this implies the result.
\end{proof}

\begin{proof}[Proof of Lemma 4.2]
Suppose not.
Let $[0, \mu)$ be the maximal positive interval of existence.
Then $v$ is positive over $(0,\mu)$.
Then by Lemma \ref{critfbound}, $a$ is positive and $b$ is finite.
By Lemma \ref{critLbound}, $\mu$ is finite.
By Lemma \ref{critvbound}, the functions $v$ and $v'$ are bounded over $[0,\mu)$.
By Lemma \ref{critvsumbound}, the function
\[
	\frac{1}{|v'|^2+\lambda v^2}
\]
is bounded over $[0,\mu)$.
Therefore, the solutions $v, \theta, F, G$ can be continued beyond $\mu$, with $f$ and $|v'|^2 + \lambda v^2$ positive.
This contradicts the assumption that $[0,\mu)$ is the maximal positive interval of existence.
\end{proof}

\end{document}